\documentclass[oneside,10pt,leqno]{amsart}
\usepackage{amssymb,amsmath,latexsym,amscd}
\usepackage{array,hhline}

\usepackage{longtable}

\newcommand{\diag}{\operatorname{diag}}

\def\gg{\mathfrak{g}}
\def\gh{\mathfrak{h}}

\def\gk{\mathfrak{k}}

\def\gm{\mathfrak{m}}

\def\C{\mathbb{C}}

\def\H{\mathbb{H}}

\def\O{\mathbb{O}}

\def\Z{\mathbb{Z}}

\def\bI{\mathbf{I}}

\def\Ad{{\rm Ad}}
\def\ad{{\rm ad}\,}

\def\rank{{\rm rank}\,}

\def\diag{{\rm diag}}

\renewcommand{\thesection}{\arabic{section}}

\renewcommand{\thetable}{{\large \thesection.\arabic{equation}}}
\newtheorem{theorem}[equation]{Theorem}

\newtheorem{lemma}[equation]{Lemma}

\newtheorem{proposition}[equation]{Proposition}

\def\sideremark#1{\ifvmode\leavevmode\fi\vadjust{\vbox to0pt{\vss
 \hbox to 0pt{\hskip\hsize\hskip1em
\vbox{\hsize2cm\tiny\raggedright\pretolerance10000 
 \noindent #1\hfill}\hss}\vbox to8pt{\vfil}\vss}}} 

\begin{document}

\title{Local and Global Homogeneity for Manifolds that admit a 
Positive Curvature Metric}

\author{Joseph A. Wolf}\thanks{Research partially supported by a Simons
Foundation grant, and by hospitality from the MATRIX math research institute 
at the Creswick campus of the University of Melbourne}
\address{Department of Mathematics \\ University of California, Berkeley \\
	CA 94720--3840, U.S.A.} \email{jawolf@math.berkeley.edu}

\date{file /texdata/working/pos\_curv\_normal/pos\_curv\_norm.tex 
last edited 2 July 2019}



\begin{abstract}
In this note we study globally homogeneous Riemannian quotients
$\Gamma\backslash (M,ds^2)$ of homogeneous Riemannian manifolds 
$(M,ds^2)$.  The Homogeneity Conjecture is that $\Gamma\backslash (M,ds^2)$
is (globally) homogeneous if and only if $(M,ds^2)$ is homogeneous 
and every $\gamma \in \Gamma$ is of constant displacement on 
$(M,ds^2)$.  We provide further evidence for that conjecture by 
(i) verifying it for normal homogeneous Riemannian manifolds $(M = G/H,ds^2)$
that also admit a $G$--invariant metric of positive curvature, and
(ii) showing that in most cases the normality condition can be dropped.  
\end{abstract}

\maketitle

\section{Introduction}\label{sec1} 
\setcounter{equation}{0}

We say that a homogeneous space $M = G/H$ is of {\em positive
curvature type} if it admits a $G$--invariant Riemannian metric
of strictly positive curvature.  Here we study such spaces and 
ask when their Riemannian quotient manifolds $\Gamma\backslash (M ,ds^2)$
are globally homogeneous.  

We show that a certain conjecture, concerning global
homogeneity for locally homogeneous Riemannian manifolds
$\Gamma \backslash (M,ds^2)$ holds in general when $M = G/H$ is of 
positive curvature type and $ds^2$ is the normal Riemannian metric.
That is Theorem \ref{conj-pos-curv}.
It depends on certain classifications (see Table \ref{shankar-table}
below) and on earlier results \cite{W2018} 
concerning isotropy--split fibrations.
We then explore the possibility of dropping the
normality requirement.  Theorem \ref{conj-non-normal} eliminates 
the normality condition on the Riemannian metric for all but three of the
Riemannian manifolds of positive curvature type.
That requires a modification
(\ref{new-setup}) of the isotropy splitting condition of \cite{W2018}.

We start by describing the Homogeneity Conjecture.
Let $(M,ds^2)$ be a connected simply connected Riemannian homogeneous
space.  Let $\pi: M \to M'$ be a Riemannian covering.  In other
words $\pi: M \to M'$ is a topological covering space that is a local
isometry.  Then the base of the covering must have form $M' = 
\Gamma \backslash M$
where $\Gamma$ is a discontinuous group of isometries of $M$ such that
only the identity element has a fixed point.  Clearly $M'$, with the 
induced Riemannian metric $ds'^2$ from $\pi: M \to M'$, is locally 
homogeneous.  We ask when $(M',ds'^2)$ is globally homogeneous.

If $M' = \Gamma \backslash M$ is homogeneous then \cite{W1960} every
element $\gamma \in \Gamma$ is of constant displacement
$\delta_\gamma(x) = dist(x,\gamma x)$ on $M$.  For the normalizer 
$N_{\bI(M,ds^2)}(\Gamma)$ of $\Gamma$ in the isometry group $\bI(M,ds^2)$
induces a transitive group of isometries on $(M',ds'^2)$.  Since $\Gamma$ 
is discrete the identity component of that normalizer actually
centralizes $\Gamma$ in $\bI(M,ds^2)$, and this centralizer is transitive
on $M$.  If $x, y \in M$ and $\gamma \in \Gamma$ we write 
$y = g(x)$ with $g$ in the centralizer of $\Gamma$.  Compute $\delta_\gamma(y)
= dist(y,\gamma y) = dist(g x,\gamma g x) = dist(gx,g\gamma x) 
= dist(x,\gamma x) = \delta_\gamma(x)$.  That is the easy half of the

\noindent {\bf Homogeneity Conjecture.}
{\it Let $M$ be a connected, simply connected Riemannian homogeneous manifold
and $M \to \Gamma \backslash M$ a Riemannian covering.  Then 
$\Gamma \backslash M$ is homogeneous if and only if every $\gamma \in \Gamma$
is an isometry of constant displacement on $M$.}

The first case is implicit in the thesis of Georges Vincent 
\cite[\S 10.5]{V1947};  he noted that the linear transformations
$\diag\{R(\theta), \dots , R(\theta)\}, R(\theta) = \left ( \begin{smallmatrix}
\cos(\theta) & \sin(\theta) \\ -\sin(\theta) & \cos(\theta) \end{smallmatrix}
\right )$, are of constant displacement on the sphere $S^{2n-1}$.  If
$\Gamma$ is the cyclic group of order $k$, generated by 
$\diag\{R(2\pi/k), \dots , R(2\pi/k)\}$, then it has centralizer 
$U(n)$ in $SO(2n)$ for $k > 2$, all of $SO(2n)$ for $k \leqq 2$,
so its centralizer is transitive on $S^{2n-1}$, and 
$\Gamma \backslash S^{2n-1}$ is homogeneous.
Vincent did not consider homogeneity, but he referred to such linear
transformations as {\sl Clifford translations} (``translation au sens de 
Clifford'') and examined space forms $\Gamma \backslash S^{2n-1}$ where
$\Gamma$ is a cyclic group $\langle \diag\{R(2\pi/k), \dots , R(2\pi/k)\}
\rangle$.

I extended this to a proof of the Homogeneity Conjecture, first for
spherical space forms \cite{W1961} and then for locally symmetric
Riemannian manifolds \cite{W1962}.  The proof in \cite{W1962}
used classification and case by case checking.  This was partially 
improved by Freudenthal \cite{F1963} and 
Ozols (\cite{O1969}, \cite{O1973}, \cite{O1974}), who gave direct proofs 
for the case where $\Gamma$ is contained in the identity component 
of $\bI(M,ds^2)$.

Since then a number of special cases of the Homogeneity Conjecture have
been verified.  The ones I know about are the case \cite{W1964} where 
$M$ is of non--positive sectional curvature, extended by Druetta \cite{D1983}
to the case where $M$ has no focal points; C\' ampoli's work
(\cite{C1983}, \cite{C1986}) on the case
where $M$ is a Stieffel manifold and $ds^2$ is the normal Riemannian metric; 
the case \cite{MMW1986} where $M$ admits a transitive semisimple group 
of isometries that has no compact factor; the case \cite{W2016} where
$M$ admits a transitive solvable group of isometries; the case \cite{W2018} 
where $M$ has a fibration such as that of Stieffel manifolds 
over Grassmann manifolds;
and the case \cite{XW2016} where every element of $\Gamma$ is close to the
identity and $M$ belongs to a certain class of Riemannian normal homogeneous 
spaces.  incidentally, the Homogeneity Conjecture is valid for locally symmetric 
Finsler manifolds as well \cite{DW2013}.

There has also been a lot of work on the infinitesimal version of constant
displacement isometries. Those are the Killing vector fields of constant 
length.  For
example see papers (\cite{BN2008a}, \cite{BN2008b}, \cite {BN2009}) of 
Berestovskii and Nikonorov, and, in the Finsler manifold setting, 
of Deng and Xu (\cite{DX2013}, \cite{DX2014a}, \cite{DX2014b}, \cite{DX2014c}, 
\cite{DX2015a}, \cite{DX2015b}, \cite{DX2016}).

In this note we verify the Homogeneity Conjecture for (i) the case where
$(M,ds^2)$ is a normal homogeneous Riemannian manifold of positive 
curvature type, (ii) most cases where $(M,ds^2)$ is a homogeneous Riemannian
manifold, not necessarily normal, of positive curvature type.  The three 
cases from which we have not yet eliminated the normality requirement are
the odd dimensional 
spheres $M = G/H = SU(m+1)/SU(m), Sp(m+1)/Sp(m), \text{ and } SU(2)/\{1\}$
with possibly non--standard metrics.  In effect, this is a progress report.

I thank Wolfgang Ziller for referring me to his survey \cite{Z2007}, which
clarified the exposition.

\section{The Classification for Positive Curvature}\label{sec2}
\setcounter{equation}{0}

The connected simply connected 
homogeneous Riemannian manifolds of positive sectional curvature were
classified by Marcel Berger \cite{B1961}, Nolan Wallach \cite{W1972},
Simon Aloff and Nolan Wallach \cite{AW1974}, and
Lionel B\' erard-Bergery \cite{B1976}.  Their isometry groups were worked
out by Krishnan Shankar \cite{S2001}.  The spaces and the isometry
groups are listed in the first two columns of Table \ref{shankar-table}
below.  The third column lists some fibrations that will be relevant to our
verification of the Homogeneity Conjecture for manifolds of positive 
curvature type.  See \cite[Section 4]{Z2007} for a description of exactly which
invariant metrics have positive sectional curvature.

\addtocounter{equation}{1}
{\footnotesize
\begin{longtable}{|r|l|l|c|}
\caption*{\bf {\normalsize Table} \thetable \quad {\normalsize
        Isometry Groups of CSC Homogeneous Manifolds of Positive
	Curvature \\ \phantom{XXXXXXXXXXXXXX} 
	And Fibrations over Symmetric Spaces}} 
\label{shankar-table} \\
\hline
 & $M = G/H$ & $\bI(M,ds^2)$ & $G/H \to G/K$  \\ \hline
\hline
\endfirsthead
\multicolumn{4}{l}{{\normalsize \textit{Table \thetable\, continued from
        previous page $ \dots$}}} \\
\hline
 & $M = G/H$ & $\bI(M,ds^2)$ & $G/H \to G/K$ \\ \hline
\hline
\endhead
\hline \multicolumn{4}{r}{{\normalsize \textit{$\dots$ Table \thetable\,
        continued on next page}}} \\
\endfoot
\hline
\endlastfoot
\hline
{\rm 1} & $S^n = SO(n+1)/SO(n)$ & $O(n+1)$ &
        \\ \hline 
{\rm 2}  & $P^m(\C) = SU(m+1)/U(m)$ & $PSU(m+1)\rtimes \Z_2$ &
        \\ \hline
{\rm 3} & $P^k(\H) = Sp(k+1)/(Sp(k)\times Sp(1))$ & $Sp(k+1)/\Z_2)$ &
        \\ \hline
{\rm 4} & $P^2(\O) = F_4/Spin(9)$ & $F_4$ &
        \\ \hline
{\rm 5} & $S^6 = G_2/SU(3)$ & $O(7)$ &
        \\ \hline
{\rm 6} & $P^{2m+1}(\C) = Sp(m+1)/(Sp(m)\times U(1))$ & 
		$(Sp(m+1)/\Z_2)\times \Z_2$ & $P^{2m+1}(\C) \to P^m(\H)$
         \\ \hline
{\rm 7} & $F^6 = SU(3)/T^2$ & $(PSU(3) \rtimes \Z_2) \times \Z_2$ &
		$F^6 \to P^2(\C)$
         \\ \hline
{\rm 8} & $F^{12} = Sp(3)/(Sp(1)\times Sp(1)\times Sp(1))$ &
		$(Sp(3)/\Z_2) \times \Z_2$ & $F^{12} \to P^2(\H)$
        \\ \hline
{\rm 9} & $F^{24} = F_4/Spin(8)$ & $F_4$ & $F^{24} \to P^2(\O)$
	\\ \hline
{\rm 10} & $M^7 = SO(5)/SO(3)$ & $SO(5)$ & 
        \\ \hline
{\rm 11} & $M^{13} = SU(5)/(Sp(2) \times_{\Z_2} U(1)$ & $PSU(5) \rtimes \Z_2$ &
		$M^{13} \to P^4(\C)$
        \\ \hline
{\rm 12} & $N_{1,1} = (SU(3) \times SO(3))/U^*(2)$ & 
		$(PSU(3) \rtimes \Z_2) \times SO(3)$ & 
        \\ \hline
{\rm 13} & $\begin{array}{l} N_{k,\ell} = SU(3)/U(1)_{k,\ell} \\
	   (k,\ell) \ne (1,1), 3|(k^2+\ell^2+k\ell) \end{array}$ &
		$(PSU(3)\rtimes \Z_2) \times (U(1)\rtimes \Z_2)$ &
		$N_{k,\ell} \to P^2(\C)$
        \\ \hline
{\rm 14} & $\begin{array}{l} N_{k,\ell} = SU(3)/U(1)_{k,\ell} \\
	   (k,\ell) \ne (1,1), 3\not| (k^2+\ell^2+k\ell) \end{array}$ &
		$U(3)\rtimes \Z_2$ &
		$N_{k,\ell} \to P^2(\C)$
        \\ \hline
{\rm 15} & $S^{2m+1} = SU(m+1)/SU(m)$ & $U(m+1)\rtimes\Z_2$ & 
		$S^{2m+1} \to P^m(\C)$
        \\ \hline
{\rm 16} & $S^{4m+3} = Sp(m+1)/Sp(m)$ & $Sp(m+1)\rtimes_{\Z_2} Sp(1)$ &
		$S^{4m+3} \to P^m(\H)$
        \\ \hline
{\rm 17} & $S^3 = SU(2)$ & $O(4)$ & $S^3 \to P^1(\C) = S^2$
        \\ \hline
{\rm 18} & $S^7 = Spin(7)/G_2$ & $O(8)$ & 
        \\ \hline
{\rm 19} & $S^{15} = Spin(9)/Spin(7)$ & $Spin(9)$ & $S^{15} \to S^8$
        \\ \hline
\end{longtable}
}

Most of the embeddings $H \hookrightarrow G$ in Table \ref{shankar-table} 
are obvious, but a few might need explanation.  For (9),
$Spin(8) \hookrightarrow Spin(9) \hookrightarrow F_4$. For (10),
$SO(3) \hookrightarrow SO(5)$ is the irreducible representation of
highest weight $4\lambda$ where $\lambda$ is the fundamental highest weight;
the tangent space representation is the irreducible representation of 
highest weight $6\lambda$. For (11), $Sp(2) \hookrightarrow SU(4)$ so
$Sp(2) \times U(1)$ maps to $U(5)$ with kernel $\{(I,1), (-I,-1)\}$ and 
image in $SU(5)$. The $\Z_2$ for the non--identity component of
$\bI(M,ds^2)$ here corresponds to complex conjugation on $SU(5)$. 
For (12),  $U^*(2)$ is the image of 
$U(2) \hookrightarrow (SU(3) \times SO(3))$, given by
$h \mapsto (\alpha(h),\beta(h))$ where $\alpha(h) = 
\left ( \begin{smallmatrix} h & 0 \\ 0 & 1/\det(h) \end{smallmatrix} \right )$
and $\beta$ is the projection $U(2) \to U(2)/(center) \cong SO(3)$.
For (13) and (14), $U(1) = H \hookrightarrow G = SU(3)$ is $e^{i\theta} 
\mapsto \diag\{e^{ik\theta},e^{i\ell \theta},e^{-i(k+\ell)\theta}\}$.
For (19), $Spin(7) \hookrightarrow Spin(8) \hookrightarrow
Spin(9)$.

Note that the first four spaces $M = G/H$ of Table 
\ref{shankar-table} are Riemannian symmetric spaces with $G = \bI(M)^0$. 
The fifth space is $S^6 = G_2/SU(3)$, where the isotropy group $SU(3)$ is 
irreducible on the tangent space, so the only invariant metric is the one
of constant positive curvature; thus it is isometric
to a Riemannian symmetric space.  In view of \cite{W1962}, 
\begin{proposition}\label{entries1-5}
The Homogeneity Conjecture is valid for the entries {\rm (1)} through
{\rm (5)} of {\rm Table \ref{shankar-table}}.
\end{proposition}

\section{Positive Curvature and Isotropy Splitting}\label{sec3}
\setcounter{equation}{0}

Some of the entries of Table \ref{shankar-table} are of
positive Euler characteristic, i.e. have $\rank H = \rank G$. In
those cases every element of $G = \bI(M,ds^2)^0$ is conjugate to
an element of $H$, hence has a fixed point on $M$.  Those are the
table entries (6), (7), (8) and (9).  Each is isotropy--split with
fibration over a projective (thus Riemannian symmetric) space, as defined in 
\cite[(1.1)]{W2018}.  The Homogeneity Conjecture follows, for 
these $(M,ds^2)$ where $ds^2$ is the normal Riemannian metric 
\cite[Corollary 5.7]{W2018}: 
\begin{proposition}\label{entries6-9}
The Homogeneity Conjecture is valid for the entries {\rm (6)} through
{\rm (9)} of {\rm Table \ref{shankar-table}}, where $ds^2$ is the normal
Riemannian metric on $M$.
\end{proposition}

The argument of Proposition \ref{entries6-9} applies with only obvious
changes to a number of other table entries, using \cite[Theorem 6.1]{W2018}
instead of \cite[Corollary 5.7]{W2018}.  That gives us
\begin{proposition}\label{many-split-iso}
The Homogeneity Conjecture is valid for the entries
{\rm (11), (13), (14), (15), (16), (17) and (19)} of
{\rm Table \ref{shankar-table}}, where $ds^2$ is the normal
Riemannian metric on $M$.
\end{proposition}

\section{The Three Remaining Positive Curvature Cases}\label{sec4}
\setcounter{equation}{0}

In positive curvature it remains only to verify the Homogeneity Conjecture 
for table entries $M = G/H$ given by (10) $M^7 = SO(5)/SO(3)$, 
(12) $N_{1,1} = (SU(3)\times SO(3))/U^*(2)$, and
(18) $S^7 = Spin(7)/G_2$\,.  For (10) and (18), $H$ is irreducible on the 
tangent space of $M$.  In particular for (18) $ds^2$ must be the constant
positive curvature metric on $S^7$, where the Homogeneity Conjecture is known:

\begin{lemma}\label{entry-18}
The Homogeneity Conjecture is valid for entry {\rm (18)}
of {\rm Table \ref{shankar-table}}, where $ds^2$ is the normal
Riemannian metric on $M$.
\end{lemma}

Now consider the case $M^7 = G/H = SO(5)/SO(3)$.  There $H$ acts
irreducibly on the tangent space, so the only invariant metric is a
normal one.  Let $\gg = \gh + \gm$ where $\gh \perp \gm$ and $\gm$
represents the tangent space at $x_0 = 1H \in G/H$.  If $\eta \in \gm$
then $t \mapsto \exp(t\eta)$ is a geodesic based at $x_0$.

Suppose that we have an isometry $\gamma$ of some
constant displacement $d > 0$.  As $\bI(M,ds^2) = SO(5)$ is
connected we have $\xi \in \gm$ such that $\sigma(t) = \exp(t\xi)x_0\,,
0 \leqq t \leqq 1$, is a minimizing geodesic in $M$ from $x_0$ to 
$\gamma(x_0)$.  Let $X$ denote the Killing vector field on $M$ 
corresponding to $\xi$.  Note that $||X_{x_0}|| = ||\xi|| = d$.
Let $g \in G$ and $y = gx_0 \in M$.  Then $t \mapsto g\sigma(t) =
g\exp(t\xi)x_0$ is a minimizing geodesic in $M$ from $y$ to $g\gamma(x_0) = 
\Ad(g)(\gamma) (y)$.  Since $\Ad(g)(\gamma)$ has the same constant
displacement $d$ as $\gamma$ we have $||(g_*X)_{x_0}|| = d$.
But $||(g_*X)_{x_0}|| = ||X_y||$, in other words
$||X_y|| = ||X_{x_0}||$.
Thus $X$ is a Killing vector field of constant length on $SO(5)/SO(3)$.
There is no such nonzero vector field \cite{XW2016}, so $\gamma$
does not exist.  We have proved

\begin{lemma}\label{entry-10}
There is no isometry $\ne 1$ of constant displacement on the manifold $M^7$
with normal Riemannian metric.
In particular the Homogeneity Conjecture is valid for entry {\rm (10)}
of {\rm Table \ref{shankar-table}}, where $ds^2$ is the normal
Riemannian metric on $M^7$.
\end{lemma}

Finally, consider the case (12) of $N_{1,1} = G/H = 
(SU(3)\times SO(3))/U^*(2)$\,.  
Let $\gamma$ be an isometry of constant displacement $d > 0$ on
$N_{1,1}$\,.  Suppose that $\gamma^2$ also is an isometry of constant
displacement.  The argument of Lemma \ref{entry-18} shows that $\gamma$
cannot belong to the identity component $G = SU(3)\times SO(3)$ of 
$\bI(N_{1,1})$.  Further, $\gamma^2$ belongs to that identity
component, so the argument of Lemma \ref{entry-18} shows that
$\gamma^2 = 1$.

Now $\gamma = (g_1,g_2)\nu$ where $g_1 \in SU(3)$, $g_2 \in SO(3)$,
$\nu^2 = 1$, $\Ad(\nu)$ is complex conjugation on $SU(3)$, and
$\Ad(\nu)$ is the identity on $SO(3)$.  The centralizer of $\nu$ 
is $K := ((SO(3) \times SO(3)) \cup (SO(3) \times SO(3))\nu$.  
Let $T_1$ (resp. $T_2$) be a maximal torus of the first (resp. second) 
$SO(3)$.  Following de Siebenthal \cite{dS1956} we may assume $g_i \in T_i$
where we replace $\gamma$ be a conjugate.  Compute
$\gamma^2 = (g_1,g_2)(\overline{g_1},g_2) = (g_1^2,g_2^2)$.
We have reduced our considerations to the cases where $g_1$ is either
the identity matrix $I_3$ or the matrix $I'_3 := 
\left ( \begin{smallmatrix} -1 & 0 & 0 \\ 0 & -1 & 0 \\ 0 & 0 & +1 
\end{smallmatrix} \right )$, and also $g_2$ is either $I_3$ or $I'_3$. 

Recall that $H = U^*(2)$ is the image of
$U(2) \hookrightarrow (SU(3) \times SO(3))$, given by
$h \mapsto (\alpha(h),\beta(h))$ where $\alpha(h) = 
\left ( \begin{smallmatrix} h & 0 \\ 0 & 1/\det(h) \end{smallmatrix} \right )$
and $\beta$ is the projection $U(2) \to U(2)/(center) \cong SO(3)$.
Further, $\bI(L_{1,1}) = G \cup G\nu$
and its isotropy subgroup is $H \cup H\nu$.  Observe that

\qquad if $(g_1,g_2) = (I_3\,, I_3)$ then $\gamma = 
	(\alpha(I_2),\beta(I_2))\nu \in H\nu$, and

\qquad if $(g_1,g_2) = (I'_3\,, I_3)$ then $\gamma = 
        (\alpha(-I_2),\beta(-I_2))\nu \in H\nu$.

\noindent
We can replace $I'_3$ by its $SO(3)$--conjugate $I''_3 =
\left ( \begin{smallmatrix} -1 & 0 & 0 \\ 0 & +1 & 0 \\ 0 & 0 & -1 
\end{smallmatrix} \right )$.  Set $I''_2 =
\left ( \begin{smallmatrix} -1 & 0 \\ 0 & +1 \end{smallmatrix} \right )$. 
Note that

\qquad if $(g_1,g_2) = (I''_3\,, I''_3)$ then $\gamma = 
        (\alpha(I''_2),\beta(I''_2))\nu \in H\nu$.

\noindent
When $\gamma \in H\nu$ it cannot be of nonzero constant displacement. 
We have reduced our considerations to the case $\gamma = (I_3\,, I'_3)\nu$,
or equivalently to one of its conjugates.  Compute

$\left ( \left ( \begin{smallmatrix} i & 0 & 0 \\ 0 & i & 0 \\ 0 & 0 & -1  
\end{smallmatrix} \right ), I_3\right ) \cdot (I_3\,, I'_3)\nu \cdot
\left ( \left ( \begin{smallmatrix} i & 0 & 0 \\ 0 & i & 0 \\ 0 & 0 & -1  
\end{smallmatrix} \right ), I_3 \right )^{-1} =
\left ( \left ( \begin{smallmatrix} i & 0 & 0 \\ 0 & i & 0 \\ 0 & 0 & -1  
\end{smallmatrix} \right ), I_3 \right ) \cdot (I_3\,, I'_3) \cdot
\left ( \left ( \begin{smallmatrix} i & 0 & 0 \\ 0 & i & 0 \\ 0 & 0 & -1  
\end{smallmatrix} \right ), I_3 \right )  \nu =
(I'_3\,, I'_3)\nu  \in H\nu$,

\noindent
so again $\gamma$ cannot be of nonzero constant displacement.
We have proved

\begin{lemma}\label{entry-12}
There is no isometry $\gamma \ne 1$ of constant displacement on the 
manifold $N_{1,1}$ with normal Riemannian metric, 
for which $\gamma^2$ also is of constant displacement.
In particular the Homogeneity Conjecture is valid for entry {\rm (12)}
of {\rm Table \ref{shankar-table}} with normal Riemannian metric on $N_{1,1}$\,.
\end{lemma}

Combining Lemma \ref{entry-18}, \ref{entry-10} and \ref{entry-12} we have
\begin{proposition}\label{entries18-12-10}
The Homogeneity Conjecture is valid for the entries {\rm (10), (12)} and
{\rm (18)} of {\rm Table \ref{shankar-table}}, where $ds^2$ is the normal
Riemannian metric on $M$.
\end{proposition}

Finally, combine Propositions \ref{entries1-5}, \ref{entries6-9},
\ref{many-split-iso} and \ref{entries18-12-10} to obtain our main
result:

\begin{theorem}\label{conj-pos-curv}
Let $(M,ds^2)$ be a connected, simply connected normal homogeneous
Riemannian manifold of positive curvature type.  Then
the Homogeneity Conjecture is valid for $(M,ds^2)$.
\end{theorem}

\section{Dropping Normality in Positive Curvature}\label{sec5}
\setcounter{equation}{0}
 
In most cases one can eliminate the normality requirement
of Theorem \ref{conj-pos-curv}.  Of course this is automatic
when the adjoint action of $H$ on the tangent space $\gg/\gh$
is irreducible; there every invariant Riemannian metric on $M = G/H$
is normal, so Theorem \ref{conj-pos-curv} applies.  Those are
the spaces given by the entries (1) (2), (3), (4), (5), (10), 
(11), (12) and (18) of Table \ref{shankar-table}.  Some other
cases require an extension of certain results from \cite{W2018}
concerning normal Riemannian homogeneous spaces.  

In \cite{W2018} I verified the Homogeneity Conjecture for
a class of normal Riemannian homogeneous
spaces $M = G/H$ that fiber over homogeneous spaces $M' = G/K$
where $H$ is a local direct factor of $K$.  
We are going to weaken the normality conditions in 
such a way that the results still apply to some of the homogeneous Riemannian 
manifolds of positive curvature type.  In \cite{W2018} the metrics 
on $M$ and $M'$ were required to be the
normal Riemannian metrics defined by the Killing form of $G$.
Instead, we look at Riemannian surjections $\pi : M \to M'$, with fiber $F$,
as follows.

\begin{equation}\label{new-setup}
\begin{aligned}
&G \text{ is a compact connected simply connected Lie group,}\\
&H \subset K \text{ are closed connected subgroups of $G$ such that } \\
& \phantom{XXXX}\text{(i) $M = G/H$\,, $M' = G/K$, and $F=H\backslash K$\,,} \\
& \phantom{XXXX}\text{(ii) $\pi: M \to M'$ is given by  $\pi(gH) = gK$\,,
                right action of $K$ on $G/H$\,, } \\
& \phantom{XXXX}\text{(iii) $M'$ and $F$ are Riemannian symmetric spaces, and}\\
& \phantom{XXXX}\text{(iv) the tangent spaces $\gm'$ for $M'$, $\gm''$
        for $F$ and $(\gm' + \gm'')$ for $M$ satisfy $\gm' \perp \gm''$}\,.
\end{aligned}
\end{equation}

We first modify \cite[Lemma 5.2]{W2018}:

\begin{lemma}\label{new-go-space} Assume {\rm (\ref{new-setup})}.
Then the fiber $F$ of
$M \to M'$ is totally geodesic in $M$
In particular it is a geodesic orbit space, and any geodesic of
$M$ tangent to $F$ at some point is of the form
$t \mapsto \exp(t\xi)x$ with $x \in F$ and $\xi \in \gm''$\,.
\end{lemma}

\begin{proof}
Restrict the adjoint representation of $\gg$ to $\gh$\,.  Then
$\gg = \gh + \gm' + \gm''$ where $\gh$ acts by its adjoint representation
(on itself), by the restriction of the isotropy representation of
$\gk$ on $\gm'$, and by its isotropy representation on $\gm''$.
If $\xi \in \gm'' (\subset \gk)$
and $\eta \in \gm' (\subset \gg)$ then $[\xi,\eta] \in \gm'$, so
$\langle [\xi,\eta]_{\gm' + \gm''},\xi\rangle = 0$ because
$\langle \xi,\gm'\rangle = 0$.  If $\xi, \eta \in \gm''$ then
$\ad(\eta)$ is antisymmetric so $[\xi,\eta] \perp \xi$\,, i.e.
$\langle [\xi,\eta]_{\gm' + \gm''},\xi\rangle = 0$.  We have just shown that
if $\xi \in \gm''$ then $t \mapsto \exp(t\xi)H$ is a geodesic in
$M$ based at $1H$.  Since it is a typical geodesic in the symmetric space
$F = H\backslash K$ we conclude that $F$ is totally geodesic in $M$.
\end{proof}

Now we set up the conditions for applying (\ref{new-setup}) in the 
possible absence of normality.

\begin{lemma}\label{apply-setup}
Suppose that $(M,ds^2)$ is one of the entries of\, 
{\rm Table \ref{shankar-table}} for which $G = \bI(M,ds^2)^0$,
that $(M,ds^2)$ satisfies {\rm (\ref{new-setup})}, and that
$\pi: (M,ds^2) \to (M',ds'^2)$ is a Riemannian submersion.  
Then the Homogeneity Conjecture holds for $(M,ds^2)$.
\end{lemma}

\begin{proof}  Running over the last column of Table \ref{shankar-table}
we see something surprising: in each case $\rank G = \rank K$, in other
words $\chi(M') > 0$.  That noted, we use the starting point of the 
proof of \cite[Proposition 5.4]{W2018}.  There one takes $\gamma \in G$ of 
constant displacement on $M$ to be of the
form $(g,r(k))$ with $g \in G$ acting on the left on $M = G/H$ and
$r(k)$ given by the right action of the normalizer of $H$ in $G$.
The metric $ds'^2$ is the usual one in each case because it is
$G$--invariant.  The argument of \cite[Proposition 5.4]{W2018} adapts to 
show that $\Gamma \cap G$ is central in $G$. Further, \cite[Lemma 5.5]{W2018}
goes through for the intersection of $\Gamma$ with other components
of $\bI(M,ds^2)$.  Thus, if $\Gamma$ is a group of isometries of constant
displacement on $(M,ds^2)$, then $\Gamma$ 
centralizes $G$.  The Homogeneity Conjecture follows for $(M,ds^2)$.
\end{proof}

The argument of Lemma \ref{apply-setup} does not apply directly to
entry (6) of Table \ref{shankar-table}, so we give another argument
specific to that case.  
There $\rank H = \rank G$ so every element of
$G$ has a fixed point on $M$.  If $\Gamma \ne \{1\}$ is a subgroup
of $\bI(M,ds^2)$ acting freely on $M$ then $\Gamma = \{1,g\nu\}$ where
$g \in G$ and $\nu$ is complex conjugation.  As $(g\nu)^2 = 1$ we have
$g\cdot {^tg}^{-1} = g\overline{g} = \pm I$ in terms of matrices.
Then ${^tg} = cg$ for some $c \in \C$ and $g = c^2g$ so $c = \pm 1$.
If $c = 1$ then $g = {^t}g$.  In that case $g$ is diagonalized by a
real matrix and $g\nu$ has a fixed point.  Thus $c = -1$ and we may assume
$g = \left ( \begin{smallmatrix} 0 & I \\ -I & 0 \end{smallmatrix} \right )$.
Thus $g\nu$ centralizes $G$, and that proves the Homogeneity Conjecture
for case (6) of Table \ref{shankar-table}.

A small modification of the argument of Lemma \ref{apply-setup} applies
to entry (19), where $G = \bI(M,ds^2)$.  We simply note that the proof
of \cite[Theorem 6.1]{W2018} proves the Homogeneity Conjecture
for case (19) of Table \ref{shankar-table}.

Now we run over the entries of Table \ref{shankar-table}.  As noted above,
in cases (1), (2), (3), (4), (5), (10), (11), (12) and (18) the group $H$
is irreducible on the tangent space of $M$, so the metric is the normal
Riemannian metric, and the Homogeneity Conjecture follows from the normal
metric case, Theorem \ref{conj-pos-curv}.  And we just dealt with entries
(6) and (19).  Now we run through the other entries.

In table entries (7), (8) and (9), $G = \bI(M,ds^2)^0$, and in entries
(13) and (14) we may assume $G = U(3) = \bI(M,ds^2)^0$.  For all of
those $\gm' \perp \gm''$ because the representations of $H$ on those
spaces have no common summand, and $\pi: (M,ds^2) \to (M',ds'^2)$ is a 
Riemannian submersion because $(M',ds'^2)$ is an irreducible Riemannian
symmetric space.  Thus Lemma \ref{apply-setup} applies to (7), (8),
(9), (13) and (14).

Consider the table entries for which $M$ is an odd sphere.  Cases (1),
(18) and (19) have already been dealt with, leaving (15), (16) and (17).
Combining all the results of this section we have

\begin{theorem}\label{conj-non-normal}
Let $(M,ds^2)$ be a connected, simply connected homogeneous Riemannian
manifold of strictly positive curvature, where $ds^2$ is not required
to be the normal Riemannian metric.  Suppose that $(M,ds^2)$ is not
one of the entries {\rm (15), (16)} or {\rm (17)} of 
{\rm Table \ref{shankar-table}}.  Then the Homogeneity Conjecture
is valid for $(M,ds^2)$.
\end{theorem}


\begin{thebibliography}{99}

\bibitem{AW1974}
S. Aloff \& N. R. Wallach,
An infinite family of distinct $7$--manifolds admitting positively curved
Riemannian structure, Bull. Amer. Math. Soc. {\bf 81} (1975), 93--97.

\bibitem{B1976}
L. B\' erard-Bergery, 
Les vari\' et\' es riemanniennes homog\` enes simplement connexes de 
dimension impaire \` a courbure strictement positive, J. Math. Pures Appl. 
{\bf 55} (1976), 47--67.

\bibitem{BN2008a} 
V. N. Berestovskii \& Y. G. Nikonorov, 
Killing vector fields of constant length on locally symmetric Riemannian 
manifolds. Transformation Groups, {\bf 13} (2008), 25--45. 

\bibitem{BN2008b} 
V. N. Berestovskii \& Y. G. Nikonorov, 
On Clifford--Wolf homogeneous Riemannian manifolds, 
Doklady Math., {\bf 78} (2008), 807--810. 

\bibitem{BN2009} 
V. N. Berestovskii \& Y. G. Nikonorov, 
Clifford--Wolf homogeneous Riemannian manifolds. J. Differential Geom. 
{\bf 82} (2009), 467--500.

\bibitem{B1961}
M. Berger, 
Les vari\' et\' es riemanniennes homog\` enes normales simplement connexes 
\` a courbure strictement positive,  Ann. Scuola Norm. Sup. Pisa 
{\bf 15} (1961), 179--246.

\bibitem{C1983}  
O. C\' ampoli, 
Clifford isometries of compact homogeneous Riemannian manifolds, 
Revista Uni\' on Math. Argentina {\bf 31} (1983), 44--49. 

\bibitem{C1986}
O. C\' ampoli,
Clifford isometries of the real Stieffel manifolds, 
Proc. Amer. Math. Soc. {\bf 97} (1986), 307--310.

\bibitem{DW2013} 
S. Deng \& J. A. Wolf, 
Locally symmetric homogeneous Finsler spaces,
International Mathematical Research Notes (IMRN) {\bf 2013} (2012),
4223--4242. 

\bibitem{DX2013} 
S. Deng \& M. Xu,
Clifford--Wolf Homogeneous Randers spaces,
J. Lie Theory {\bf 23} (2013), 837--845. 

\bibitem{DX2014a}
S. Deng \& M. Xu,
Clifford--Wolf Homogeneous Randers spheres, Israel J. Math., {\bf 199} 
(2014), 507--525.

\bibitem{DX2014b}
S. Deng \& M. Xu,
Clifford--Wolf translations of left invariant Randers metrics on compact 
Lie groups, Quarterly J. Math.  {\bf 65} (2014), 133--148.

\bibitem{DX2014c}
S. Deng \& M. Xu, 
Clifford--Wolf translations of Finsler spaces. Forum Math. {\bf 26} 
(2014), 1413--1428. 

\bibitem{DX2015a} 
S. Deng \& M. Xu,
Left invariant Clifford--Wolf homogeneous $(\alpha,\beta)$--metrics on 
compact semisimple Lie groups. Transform. Groups {\bf 20} (2015), 395--416. 

\bibitem{DX2015b} 
S. Deng \& M. Xu,
Clifford--Wolf homogeneous Finsler metrics on spheres,
Annali di Matematica Pura ed Applicata {\bf 194} (2015), 759--766.

\bibitem{DX2016}
S. Deng \& M. Xu,
$(\alpha_1,\alpha_2)$--metrics and Clifford--Wolf homogeneity. J. Geom. 
Anal. {\bf 26} (2016), 2282--2321.

\bibitem{MMW1986} 
I. Dotti Miatello, R. J. Miatello \& J. A. Wolf,
Bounded isometries and homogeneous Riemannian quotient manifolds,
Geometriae Dedicata, {\bf 21} (1986), 21--28.

\bibitem{D1983} 
M. J. Druetta, 
Clifford translations in manifolds without focal points, 
Geometriae Dedicata, {\bf 14} (1983), 95--103.

\bibitem{F1963}
H. Freudenthal, 
Clifford--Wolf-Isometrien symmetrischer
R\" aume, Math. Ann. {\bf 150} (1963), 136--149.

\bibitem{O1969}
V. Ozols, 
Critical points of the displacement function
of an isometry, J. Differential Geometry {\bf 3} (1969), 411--432.

\bibitem{O1973}
V. Ozols,
Critical sets of isometries, Proc. Sympos. Pure Math.
{\bf 27} (1973), 375--378.

\bibitem{O1974}
V. Ozols,
Clifford translations of symmetric spaces,
Proc. Amer. Math. Soc. {\bf 44} (1974), 169--175.

\bibitem{S2001}
K. Shankar,
Isometry groups of homogeneous spaces with positive sectional curvature,
Diff. Geom. Appl. {\bf 14} (2001), 57--78.

\bibitem{dS1956}
J. de Siebenthal, 
Sur les groupes de Lie compacts non connexes, Comment. Math. Helv. 
{\bf 31} (1956), 41--89.

\bibitem{V1947}
G. Vincent,
Les groupes lin\' eaires finis sans points fixes,
Commentarii Mathematici Helvetici {\bf 20} (1947), 117--171.

\bibitem{W1972}
N. R. Wallach, 
Compact homogeneous Riemannian manifolds with strictly positive curvature,
Ann. Math. {\bf 96} (1972), 277--295.

\bibitem{W1960}
J. A. Wolf,
Sur la classification des vari\' et\' es riemanni\`ennes
homog\` enes \` a courbure constante, C. R. Acad. Sci. Paris,
{\bf 250} (1960), 3443--3445.

\bibitem{W1961}
J. A. Wolf,
Vincent's conjecture on Clifford translations of the sphere,
Commentarii Mathematici Helvetici, {\bf 36} (1961), 33--41.

\bibitem{W1962}
J. A. Wolf,
Locally symmetric homogeneous spaces.  Commentarii
Mathematici Helvetici, {\bf 37} (1962), 65--101.

\bibitem{W1964}
J. A. Wolf,
Homogeneity and bounded isometries in manifolds of
negative curvature, Illinois Journal of Mathematics, {\bf 8} (1964), 14--18.

\bibitem{W2016} 
J. A. Wolf, 
Bounded isometries and homogeneous quotients, Journal
of Geometric Analysis {\bf 26} (2016), 1--9.  

\bibitem{W2018} 
J. A. Wolf, 
Homogeneity for a class of Riemannian quotient manifolds,
Differential Geometry and its Applications {\bf 56} (2018), 355--372.

\bibitem{XW2016} 
M. Xu \& J. A. Wolf,
Killing vector fields of constant length on Riemannian normal homogeneous
spaces, Transformation Groups {\bf 21} (2016), 871--902.  

\bibitem{Z2007}
W. Ziller,
Examples of Riemannian manifolds with non--negative sectional curvature,
in Metric and Comparison Geometry, Surveys in Differential Geometry {\bf 11},
International Press 2007, 63--102.

\end{thebibliography}
\end{document}